\documentclass[a4paper,12pt,reqno]{amsart}
\usepackage{dsfont}
\usepackage{amsmath,amsfonts,amssymb,amsthm,enumerate,multicol}%hyperref,
\usepackage{tikz,graphicx}
\usepackage{hyperref}
\usepackage{caption,enumitem,wasysym}
\usepackage{cleveref}
\usepackage{epstopdf}
\usepackage{float,wrapfig}
\usepackage[labelsep=space]{caption}
\usepackage[font=footnotesize]{caption}
\usepackage{xcolor}
\newcommand{\sign}{\text{sign}}

\newtheorem{proposition}{Proposition}[section]
\newtheorem{definition}{Definition}[section]
\newtheorem{theorem}{Theorem}[section]

\newtheorem{lemma}{Lemma}[section]

\allowdisplaybreaks

 \theoremstyle{plain}

\usepackage{mathtools}

\usepackage{amsmath,amsfonts,amssymb,amsthm,enumerate,multicol}%hyperref,
\usepackage{tikz}
\usepackage{float}
\usepackage{pgfplots}
\pgfplotsset{compat=1.14}
\usepackage{tkz-graph}

\usepackage{autobreak}

\allowdisplaybreaks

\numberwithin{equation}{section}

\begin{document}

\begin{center}
\Large{\textbf{Existence of solutions to the continuous Redner–Ben-Avraham–Kahng coagulation equation}}
\end{center}
\medskip
%\centerline{by}
\medskip
\centerline{${\text{Pratibha~Verma $^1$}}$, ${\text{Ankik ~ Kumar ~Giri$^{1*}$}}$ }\let\thefootnote\relax\footnotetext{$^*$Corresponding author. Tel +91-1332-284818 (O);  Fax: +91-1332-273560  \newline{\it{${}$ \hspace{.3cm} Email address: }}ankik.giri@ma.iitr.ac.in}
\medskip
{\footnotesize
 %please put the address of the second  and third author

  \centerline{ ${}^{1}$Department of Mathematics, Indian Institute of Technology Roorkee,}
   %\centerline{Other lines}
   \centerline{Roorkee-247667, Uttarakhand, India}

}

\bigskip

\section*{ABSTRACT}
We take into account a coagulation model that simulates a distinct kind of dynamics. In this model, two particles collide to produce a single particle, but the resulting particle decreases in size, allowing each particle to be fully identified by its size. It is demonstrated that the corresponding evolving integral partial differential equation has solutions for product-type coagulation kernels i.e. $0 \le \mathfrak{K}(\varrho, \varsigma)= \mathfrak{K}(\varsigma, \varrho)=r(\varsigma) r(\varrho)+\alpha(\varsigma, \varrho), (\varsigma, \varrho)\in \mathbb{R}_+^2$ and $\sup_{\varsigma \in [0,R]} \frac{\mathfrak{K}(\varsigma, \varrho)}{\varrho} \to 0 \ \mbox{as}  \ \varrho \to \infty$.

\medskip

\textbf{KeyWords :} Coagulation; Annihilation; Redner–Ben-Avraham–Kahng Coagulation System; Existence.\\

\textbf{2020 Mathematics subject classification}: 45J05, 45K05, 47G20, 34K30, 45G10.

\section{INTRODUCTION}

In many scientific fields, including polymer chemistry, colloid science, cloud dynamics, and star formation, particles grow irreversibly by successively merging clusters of particles. In the past, much research has been done on the kinetics of coagulation/aggregation models. Coagulation is a process that combine particles, molecules, or substances to produce a bigger mass or clump. This can happen physically or chemically and has been seen in a variety of domains, including biology, chemistry, and materials research. The coagulation process can be conceptually expressed as follows:

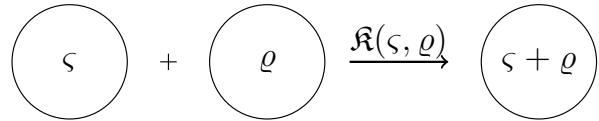
\begin{figure}[htb]
    \centering
    \resizebox{0.5\textwidth}{!}{
\begin{tikzpicture}
\draw (0,0) circle (0.7cm) node{$\varsigma$};
\draw (1.1,0) -- (1.3,0);
\draw (1.2,0.1) -- (1.2,-0.1);
\draw (2.4,0) circle (0.7cm) node{$\varrho$};
\draw[thick,->] (3.4,0) -- (4.6,0);
\draw (4,0.25) node{$\mathfrak{K}(\varsigma, \varrho)$};
\draw (5.7,0) circle (0.7cm) node{$\varsigma + \varrho$};
\end{tikzpicture}
 }
    %\caption{This figure has a width which is a factor of text width}
  \end{figure}

where $\mathfrak{K}(\varsigma, \varrho)=\mathfrak{K}(\varrho, \varsigma) \ge 0$ is the coagulation kernel, which indicates the rate at which clusters of size $\varsigma \in \mathbb{R}_+$ and $\varrho \in \mathbb{R}_+$ coalesce. We will use the notations $\mathbb{R}_+= (0, \infty)$ and $\mathbb{R}_+^2=\mathbb{R}_+ \times \mathbb{R}_+$ throughout this paper. The Smoluchowski's coagulation equation \cite{Muller, Smoluchowski} is the most widely used mean-field model for modelling the coagulation process in continuous and discrete sense respectively. 

A particle and its corresponding antiparticle colliding and converting into energy is known as the annihilation process. A particle and its antiparticle completely destroy one another when they come into contact. Particles and antiparticles have opposite charges and other quantum qualities that are equal in magnitude but opposite in sign, which causes this process to happen. The annihilation process is crucial for comprehending the behaviour of subatomic particles and is employed in many contexts, such as positron emission tomography (PET) scans for medical imaging and studies that investigate the characteristics of matter and antimatter. The following is a conceptual explanation of the annihilation process:

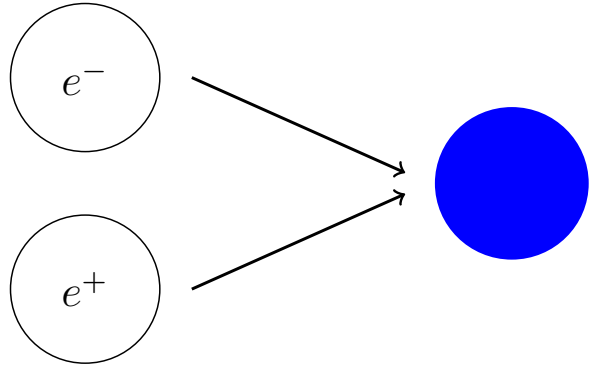
\begin{figure}[htb]
    \centering
    \resizebox{0.5\textwidth}{!}{
\begin{tikzpicture}
\draw (0,0) circle (0.7cm) node{$e^-$};
\draw (0,-2) circle (0.7cm) node{$e^+$};
\draw[thick,->] (1,0) -- (3,-0.9);
\draw[thick,->] (1,-2) -- (3,-1.1);
\filldraw[color=blue, fill=blue, very thick] (4,-1) circle (0.7cm);
\end{tikzpicture}
 }
    %\caption{This figure has a width which is a factor of text width}
  \end{figure}
  
Let us now focus our attention to the mathematical formulation of the model under consideration in this paper. To begin, we consider a closed system of particles undergoing binary collisions in which the smaller particle totally annihilates and destroys the identical quantity of the larger particle. The conceptual explanation of such type of coagulation is given in following diagram.
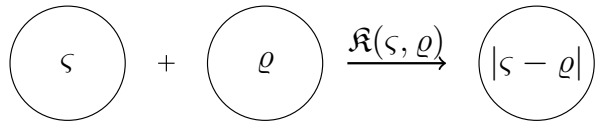
\begin{figure}[htb]
    \centering
    \resizebox{0.5\textwidth}{!}{
\begin{tikzpicture}
\draw (0,0) circle (0.7cm) node{$\varsigma$};
\draw (1.1,0) -- (1.3,0);
\draw (1.2,0.1) -- (1.2,-0.1);
\draw (2.4,0) circle (0.7cm) node{$\varrho$};
\draw[thick,->] (3.4,0) -- (4.6,0);
\draw (4,0.25) node{$\mathfrak{K}(\varsigma, \varrho)$};
\draw (5.7,0) circle (0.7cm) node{$| \varsigma - \varrho |$};
\end{tikzpicture}
 }
    %\caption{This figure has a width which is a factor of text width}
  \end{figure}
  
In such a case, the following two results are possible:
\begin{itemize}
\item production of particles of mass $\varsigma \in \mathbb{R}_+$ by coalescence of particles with masses $\varrho \in \mathbb{R}_+$ and $\varsigma+\varrho$.

\item disappearance of particles of mass $\varsigma \in \mathbb{R}_+$, due to their coalescence with particles of mass $\varrho \in \mathbb{R}_+$.
\end{itemize}

Redner et al. provided the first description of a finite dimensional model based on above concept in \cite{Redner1}. This model is known as the finite dimensional discrete Redner–Ben-Avraham–Kahng (RBK) coagulation model. In order to understand the dynamics of vicious civilizations, this model has been studied in \cite{Redner2}. Later, the existence and uniqueness of solutions for the infinite-dimensional discrete RBK coagulation model were demonstrated by da Costa et al. in \cite{dacosta1} under very general conditions on the coagulation coefficients. They also demonstrated the solutions' differentiability and their ongoing dependence on the supplied data. Additionally, several remarkable invariance properties were proven. Finally, a study of the solutions' long-term behaviour as well as an initial assessment of their scaling behaviour were carried out. The authors of \cite{dacosta2} studied the large-time behaviour of solutions to the discrete RBK coagulation system with non-negative compactly supported input data, which, due to the aforementioned invariance properties, convert this system into a finite-dimensional differential equation.

The continuous RBK coagulation equation appears as follows:
\begin{eqnarray}\label{crbk}
\frac{\partial \mathfrak{f}}{\partial \tau}=\int_0^{\infty} \mathfrak{K}(\varsigma+\varrho,\varrho) \mathfrak{f}(\varsigma+\varrho,\tau) \mathfrak{f}(\varrho,\tau) d\varrho- \int_0^{\infty} \mathfrak{K}(\varsigma,\varrho) \mathfrak{f}(\varsigma,\tau) \mathfrak{f}(\varrho,\tau) d\varrho
\end{eqnarray}
with initial condition
\begin{eqnarray}\label{ic}
\mathfrak{f}(\varsigma,0)= \mathfrak{f}_0(\varsigma)\ge 0
\end{eqnarray}
where $ \mathfrak{f}(\varsigma, \tau)$ represents the particle concentration of volume $\varsigma \in \mathbb{R}_+$ at time $\tau \ge 0$. The non-negative quantity $\mathfrak{K}(\varsigma,\varrho)$ depicts the coagulation rate at which particles of volume $\varsigma$ interact with particles of volume $\varrho$ to generate particles of volume $|\varsigma-\varrho|$. The rate `$\mathfrak{K}$' is referred to as the coagulation kernel or coagulation coefficient, which is assumed to be non-negative and symmetric i.e. $0 \le \mathfrak{K}(\varsigma, \varrho)= \mathfrak{K}(\varrho, \varsigma),~\forall~ (\varsigma, \varrho) \in \mathbb{R}_+^2$.

This paper primarily concentrates on the problem of weak solutions for the continuous RBK coagulation equation \eqref{crbk}-\eqref{ic}, which do not comply with the mass conservation condition. Several articles have been published on the subject of the existence and uniqueness of different type of solutions (weak solution, mild solution and classical solution etc.) to the Smoluchowski coagulation equation derived by employing various methodologies under diverse growth conditions on the coagulation kernel, see \cite{ball, Barik0, Barik1, Laurencot2, Laurencot1, Laurencot3, stewart1, stewart2}. In \cite{dacosta1, dacosta2}, da costa et al. addressed the discrete RBK coagulation model and presented results for existence, uniqueness, and some solution invariant features. The continuous RBK coagulation system, however, has not been extensively studied. The first result on the continuous RBK coagulation model is presented in \cite{ankik}, and it includes results on existence and uniqueness, as well as continuous dependency on beginning data and solution behaviour over a long period of time. In this article, we expanded upon the existing result from \cite{ankik}. The proofs of Theorem \ref{thmexist}, in particular, go along the lines of the existence proofs performed in \cite{Laurencot2} for the Smoluchowski coagulation equations.

The contents of this article in different sections are arranged as follows. Section 2 describes the assumptions on coagulation kernel to the existence of solution for \eqref{crbk}-\eqref{ic}. Next, in the same section we define mild solutions of \eqref{crbk}-\eqref{ic} along with the main results of the article. Section 3 is devoted to the proof of existence of at least one solution to the continuous RBK equation \eqref{crbk}-\eqref{ic} for coagulation kernel $`\mathfrak{K}'$ given in \eqref{ker}-\eqref{ker2} in space $X_{0,1}$.

\section{\textbf{SOME DEFINITION AND MAIN RESULTS}}
We make the following assumptions on $\mathfrak{K}$ which is non-negative and symmetric:
\begin{eqnarray}\label{ker}
0 \le \mathfrak{K}(\varrho, \varsigma)= \mathfrak{K}(\varsigma, \varrho)=r(\varsigma) r(\varrho)+\alpha(\varsigma, \varrho), \hspace{1.5cm}(\varsigma, \varrho)\in \mathbb{R}_+^2
\end{eqnarray}
where $r$ and $\alpha$ are non-negative functions satisfying
\begin{eqnarray}\label{ker1}
\left\{
  \begin{array}{ll}
    r \in \mathcal{C}(\mathbb{R}_+; \mathbb{R}_+),\ \ \ \ \ \ \ \alpha \in \mathcal{C}(\mathbb{R}_+^2; \mathbb{R}_+), \\
    0\le \alpha(\varsigma, \varrho)= \alpha(\varrho, \varsigma)\le Ar(\varsigma)r(\varrho),\ \ \ \ \ (\varsigma, \varrho) \in [1, +\infty)^2,
  \end{array}
\right.
\end{eqnarray}
for some positive real number $A$. Additionally, let us suppose that $\mathfrak{K}$ is strictly subquadratic, which means for each $R\ge 1$,
\begin{eqnarray}\label{ker2}
\omega_R(\varrho)=\sup_{\varsigma \in [0,R]} \frac{\mathfrak{K}(\varsigma, \varrho)}{\varrho} \to 0 \ \ \ \mbox{as} \ \ \  \varrho \to +\infty.
\end{eqnarray}
Finally, we suppose that the initial datum $\mathfrak{f}_0$ satisfies the condition
\begin{eqnarray}\label{ic1}
\mathfrak{f}_0\in X_{0,1}^+,
\end{eqnarray}
where $X_{0,1}^+$ is a positive cone of the Banach space
$$X_{0,1}=L^1((0, +\infty); (1+\varsigma) d\varsigma)$$
equipped with the norm $\|. \|_{0,1}$ as it is stated
$$\|x\|_{0,1}=\int_0^{\infty} (1+\varsigma) |x(\varsigma)| d \varsigma,\ \ \ \ \ x\in X_{0,1}.$$
Thus, $$X_{0,1}^+= \{x\in X_{0,1}, x \ge 0 \}.$$

Let us now define the notion of mild solution to the continuous RBK equation \eqref{crbk}-\eqref{ic}.

\begin{definition}\label{defmild}
Let $T\in (0,\infty]$. Assume that $\mathfrak{f}_0$ satisfies \eqref{ic1} and coagulation kernel `$\mathfrak{K}$' satisfies \eqref{ker}-\eqref{ker2}. 
A non-negative real-valued function $\mathfrak{f} : [0,T) \to X_{0,1}^+$ that satisfies \eqref{space1}, \eqref{space2} and \eqref{msol} for every $\tau\in [0,T)$ is said to be a mild solution of \eqref{crbk}-\eqref{ic} on $[0,T)$.

\begin{eqnarray}\label{space1}
\mathfrak{f}\in \mathcal{C}([0,T); L^1(0, \infty))\cap L^{\infty}(0, T; X_{0,1}),
\end{eqnarray}

\begin{eqnarray}\label{space2}
(\varrho,\mathfrak{s}) \mapsto \mathfrak{K}(\varsigma , \varrho) \mathfrak{f} (\varsigma ,\mathfrak{s}) \mathfrak{f} (\varrho ,\mathfrak{s}) \in L^1((0, \infty) \times (0,\tau)),
\end{eqnarray}
and for almost every $\varsigma \in \mathbb{R}_+$,
 \begin{eqnarray}\label{msol}
 \mathfrak{f}(\varsigma,\tau)&=&\mathfrak{f}_0(\varsigma)+\int_0^\tau \int_0^{\infty} \mathfrak{K}(\varsigma+\varrho,\varrho) \mathfrak{f}(\varsigma+\varrho,\mathfrak{s}) \mathfrak{f}(\varrho,\mathfrak{s}) d\varrho d\mathfrak{s} \nonumber\\
 & &- \int_0^\tau \int_0^{\infty} \mathfrak{K}(\varsigma,\varrho) \mathfrak{f}(\varsigma,\mathfrak{s}) \mathfrak{f}(\varrho,\mathfrak{s}) d\varrho d\mathfrak{s} .
\end{eqnarray}
\end{definition}

Our main result is given by following theorem:
\begin{theorem}\label{thmexist}
  If \eqref{ker}-\eqref{ker2} holds true for the coagulation kernel `$\mathfrak{K}$' for every $\mathfrak{f}_0\in X_{0,1}^+$, then there is at least one mild solution $\mathfrak{f}$ for \eqref{crbk}-\eqref{ic} on $[0,\infty)$ satisfying
\begin{eqnarray}\label{mass}
  \int_0^{\infty} \varsigma \mathfrak{f} (\varsigma,\tau) d\varsigma &\le & \int_0^{\infty} \varsigma \mathfrak{f}_0 (\varsigma) d\varsigma, \ \ \ \ \ \tau\in [0, \infty).
\end{eqnarray}
\end{theorem}
It follows from previous theorem that for coagulation kernal $`\mathfrak{K}'$ and initial condition $\mathfrak{f}_0$ satisfies \eqref{ker}-\eqref{ker2} and \eqref{ic1}, respectively, there exists at least one mild solution to the continuous RBK equation \eqref{crbk}-\eqref{ic} in the sense of Definition \ref{defmild}.

\medskip
\section{\textbf{PROOF OF THEOREM \ref{thmexist} (EXISTENCE OF SOLUTIONS)}}

By following a weak $L^1$-compactness method that was first introduced in the classic work of Stewart \cite{stewart1}, we are able to establish the mild solutions to the equations \eqref{crbk}-\eqref{ic}. In order to demonstrate that Theorem \ref{thmexist} is true, we must first write the equations \eqref{crbk}-\eqref{ic} into the limit of a sequence of truncated equations. These equations are obtained by exchanging the collision kernel $\mathfrak{K}$ by their cut-off kernels $\mathfrak{K}_n$.
\medskip
\subsection{The truncated model}
let $(\Xi_n)_{n\ge 1}$ such that $0 \le \Xi_n \le 1$ denote a sequence of smooth cut-off functions defined as
\begin{eqnarray*}
\Xi_n(\varsigma)=\left\{
  \begin{array}{ll}
    1, \ \  & \mbox{for} \ \ 0 \le \varsigma \le n, \\
    0,  & \ \ \  \text{elsewhere}.
  \end{array}
\right.
\end{eqnarray*}
For $n \ge 1$, we establish an approximation sequence of $\mathfrak{K}$ by
\begin{eqnarray} \label{trunker1}
\mathfrak{K}_n(\varsigma, \varrho)=\mathfrak{K}(\varsigma, \varrho) \Xi_n(\varsigma) \Xi_n(\varrho),\ \ \ (\varsigma, \varrho)\in \mathbb{R}_+^2.
\end{eqnarray}
We can easily conclude that $\mathfrak{K}_n$ is non-negative and bounded continuous function on $\mathbb{R}_+^2$ which also satisfy
\begin{eqnarray}\label{trunker2}
\mathfrak{K}_n(\varsigma, \varrho) = r_n(\varsigma) r_n(\varrho)+\alpha_n(\varsigma, \varrho),
\end{eqnarray}
where,
\begin{eqnarray}
0 \le \alpha_n(\varsigma, \varrho) \le A r_n(\varsigma) r_n(\varrho),\ \ \ \ (\varsigma, \varrho) \in[1, +\infty)^2, \label{alpha}\\
r_n(\varsigma)= r(\varsigma) \Xi_n(\varsigma), \ \ \ \alpha_n(\varsigma, \varrho) =\alpha(\varsigma, \varrho) \Xi_n(\varsigma) \Xi_n(\varrho).\label{alpha1}
\end{eqnarray}
Additionally, we examine a sequence of functions $\{ \mathfrak{f}_0^n \ge 0\}_{n\ge 1}$ in $\mathcal{D}(0, \infty)$ such that
\begin{eqnarray}\label{trunic1}
\mathfrak{f}_0^n(\varsigma)=\mathfrak{f}_0(\varsigma) \ \Xi_n(\varsigma),
\end{eqnarray}
which satisfy
\begin{eqnarray}
\lim_{n \to \infty}\|\mathfrak{f}_0^n -\mathfrak{f}_0\|_{0,1}=0, \label{trunic2}
\end{eqnarray}
and
\begin{eqnarray}
\mathbb{G}_0:= \sup_{n \ge 1} \|\mathfrak{f}_0^n\|_{0,1} <\infty. \label{trunic3}
\end{eqnarray}

By using truncation on kernel and initial condition as in \eqref{trunker1} and \eqref{trunic1}, respectively, we get the following truncated system:
\begin{eqnarray}\label{trunc}
 \frac{\partial \mathfrak{f}^n}{\partial \tau}(\varsigma,\tau)=\int_0^{n-\varsigma}\mathfrak{K}_n(\varsigma+\varrho,\varrho)\mathfrak{f}^n(\varsigma+\varrho,\tau)\mathfrak{f}^n(\varrho,\tau)d\varrho -\int_0^n \mathfrak{K}_n(\varsigma,\varrho) \mathfrak{f}^n(\varsigma,\tau) \mathfrak{f}^n(\varrho,\tau)d\varrho,~~~
 \end{eqnarray}
where $(\varsigma,\tau)\in (0,n)\times [0, +\infty)$, having initial condition
\begin{eqnarray}\label{trunic4}
\mathfrak{f}^n(\varsigma,0)=\mathfrak{f}^n_0(\varsigma), ~~~~\varsigma\in \mathbb{R}_+
\end{eqnarray}
where $n\ge 1$ is an integer.
\begin{proposition}\label{Trunprop}
Let us consider \eqref{trunker1}-\eqref{alpha1} holds and $\mathfrak{f}^n_0(\varsigma) \in X_{0,1}^+$. Then a unique non-negative solution $\mathfrak{f}^n\in \mathcal{C}(0, +\infty) \cap L^{\infty}(0, T;L^1(0,n))$ to \eqref{trunc}-\eqref{trunic4} exists which satisfies,
\begin{eqnarray}\label{mtrun}
\mathfrak{f}^n(\varsigma,\tau)&=&\mathfrak{f}_0^n(\varsigma) + \int_0^\tau \int_0^{\infty} \mathfrak{K}_n(\varsigma+\varrho, \varrho) \mathfrak{f}^n(\varsigma +\varrho, \mathfrak{s}) \mathfrak{f}^n (\varrho, \mathfrak{s}) d\varrho d\mathfrak{s}\nonumber\\
& &-\int_0^{\tau} \int_0^{\infty} \mathfrak{K}_n(\varsigma, \varrho) \mathfrak{f}^n(\varsigma, \mathfrak{s}) \mathfrak{f}^n (\varrho, \mathfrak{s}) d\varrho d\mathfrak{s},
\end{eqnarray}
for $\varsigma \in \mathbb{R}_+$, $\tau \ge 0$ and $n\ge 1$.
\end{proposition}
The proof of the above proposition is similar to that of \cite[Theorem 3.1]{ankik}.

Using the fact that the coagulation kernels $\mathfrak{K}_n$ are bounded and compactly supported in $\mathbb{R}_+$, we can conclude the following helpful identities by using \eqref{mtrun}.

Let us consider a measurable function $\phi$ on $\mathbb{R}_+$, then for $n \ge 1$ and $\tau >0$, the following holds true
\begin{eqnarray}\label{wtrun}
& & \hspace{-2.5cm}\int_0^{\infty} \phi(\varsigma)(\mathfrak{f}^n(\varsigma,\tau)-\mathfrak{f}_0^n(\varsigma)) d\varsigma \nonumber\\
&=&  \int_0^\tau \int_0^{\infty}\int_0^{\varsigma} \tilde{\phi}(\varsigma, \varrho) \mathfrak{K}_n(\varsigma, \varrho) \mathfrak{f}^n(\varsigma, \mathfrak{s}) \mathfrak{f}^n (\varrho, \mathfrak{s}) d\varrho d\varsigma d\mathfrak{s},
\end{eqnarray}
where $\tilde{\phi}(\varsigma, \varrho)= \phi(\varsigma-\varrho)-\phi(\varsigma)-\phi(\varrho)$ with $(\varsigma, \varrho) \in \mathbb{R}^2_+$, and
\begin{eqnarray}\label{masstrun}
\int_0^{\infty} \varsigma \mathfrak{f}^n(\varsigma, \tau)d\varsigma \le \int_0^{\infty} \varsigma \mathfrak{f}_0^n(\varsigma) d\varsigma.
\end{eqnarray}

\medskip
\subsection{A priori bounds}
Next, we consider that \eqref{trunker2}-\eqref{alpha1} hold then for $(\varsigma, \varrho)\in \mathbb{R}_+^2$, we have following estimates which are valid for all $n \ge 1$. Throughout this manuscript, we use some sequence of positive constants denoted by $(C_i)$. These $C_i$'s are depend on some parameter which we will indicate explicitly.

\begin{lemma}\label{l1}
Let us assume $0< T < \infty$, then for $n \ge M>0$, $\tau \in [0,T]$, the following holds
\begin{eqnarray}
\int_0^\tau \bigg(\int_M^{\infty} r_n(\varsigma) \mathfrak{f}^n( \varsigma, \mathfrak{s}) d\varsigma \bigg)^2 d\mathfrak{s} \le \frac{C_1}{M},\label{eql1.1}\\
\int_0^\infty \mathfrak{f}^n(\varsigma,\tau) d\varsigma \le \mathbb{G}_0, \label{eql1.2}\\
\int_0^\tau \bigg(\int_0^{\infty} r_n(\varsigma) \mathfrak{f}^n( \varsigma, \mathfrak{s}) d\varsigma \bigg)^2 d\mathfrak{s} \le 2\mathbb{G}_0. \label{eql1.3}
\end{eqnarray}
\end{lemma}
\begin{proof}
First, put $\phi(\varsigma)= \min\{\varsigma, M\}$ in \eqref{wtrun}, then for $\varsigma>\varrho$, we have
\begin{eqnarray*}
\tilde{\phi}(\varsigma, \varrho) =
\left\{
  \begin{array}{ll}
      -2 \varrho  & \mbox{if} \ \ \ \varsigma \in [0,M], \varrho \in [0,M] \ \mbox{and} \ (\varsigma-\varrho) \in [0,M], \\
      (\varsigma-M)-2 \varrho  & \mbox{if} \ \ \ \varsigma \in (M,+\infty),  \varrho \in [0,M] \ \mbox{and} \ (\varsigma-\varrho) \in [0,M], \\
      - \varrho  & \mbox{if} \ \ \ \varsigma \in (M,+\infty),  \varrho \in [0,M] \ \mbox{and} \ (\varsigma-\varrho) \in (M,+\infty), \\
      (\varsigma-2M)- \varrho  & \mbox{if} \ \ \ \varsigma \in (M,+\infty),  \varrho \in (M,+\infty) \ \mbox{and} \ (\varsigma-\varrho) \in [0,M], \\
    -M  & \mbox{if} \ \ \ \varsigma \in (M,+\infty),  \varrho \in (M,+\infty) \ \mbox{and} \ (\varsigma-\varrho) \in (M, +\infty),
  \end{array}
\right.
\end{eqnarray*} 
which gives
\begin{eqnarray*}
\tilde{\phi}(\varsigma, \varrho) \le
\left\{
  \begin{array}{ll}
     0  & \mbox{if} \ \ \ \varsigma \in [0,M] \ \mbox{or} \ \varrho \in [0,M], \\
    -M  & \mbox{if} \ \ \ (\varsigma, \varrho) \in (M,+\infty)^2.
  \end{array}
\right.
\end{eqnarray*} 
Now, by using the bounds on $\tilde{\phi}$, we have
\begin{eqnarray*}
& &\int_0^{\infty} \phi(\varsigma) (\mathfrak{f}^n(\varsigma,\tau)-\mathfrak{f}^n_0(\varsigma)) d\varsigma\\
& \le & -M \int_0^\tau \int_M^{\infty} \int_M^{\infty} \mathfrak{K}_n(\varsigma, \varrho) \mathfrak{f}^n(\varsigma,\mathfrak{s}) \mathfrak{f}^n(\varrho,\mathfrak{s}) d\varrho d \varsigma d\mathfrak{s}
\end{eqnarray*}
which implies
\begin{eqnarray*}
\int_0^\tau \int_M^{\infty} \int_M^{\infty} \mathfrak{K}_n(\varsigma, \varrho) \mathfrak{f}^n(\varsigma,\mathfrak{s}) \mathfrak{f}^n(\varrho,\mathfrak{s}) d\varrho d \varsigma d\mathfrak{s} \le \frac{C_1}{M}.
\end{eqnarray*}
Finally, from \eqref{trunker2} and above inequality, we obtain
\begin{eqnarray*}
& & \hspace{-2.5cm}\int_0^\tau \bigg(\int_M^{\infty} r_n(\varsigma) \mathfrak{f}^n(\varsigma, \mathfrak{s}) d\varsigma\bigg)^2d\mathfrak{s}\\
& \le & \int_0^\tau \int_M^{\infty} \int_M^{\infty} \mathfrak{K}_n(\varsigma, \varrho) \mathfrak{f}^n(\varsigma,\mathfrak{s}) \mathfrak{f}^n(\varrho,\mathfrak{s}) d\varrho d \varsigma d\mathfrak{s},
\end{eqnarray*}
which gives the proof of \eqref{eql1.1}.

Next, put $\phi \equiv 1$ in \eqref{wtrun}, then $\tilde{\phi}=-1$. In that case we obtain
\begin{eqnarray*}
\int_0^{\infty} \mathfrak{f}^n (\varsigma, \tau) d\varsigma + \int_0^\tau \int_0^{\infty}\int_0^{\varsigma} \mathfrak{K}_n(\varsigma, \varrho) \mathfrak{f}^n(\varsigma, \mathfrak{s}) \mathfrak{f}^n(\varrho, \mathfrak{s}) d\varrho d\varsigma d\mathfrak{s} =\int_0^{\infty} \mathfrak{f}^n_0 (\varsigma) d\varsigma.
\end{eqnarray*}
Above equation along with Fubini's theorem entails that
\begin{eqnarray*}
\int_0^{\infty} \mathfrak{f}^n (\varsigma, \tau) d\varsigma + \frac{1}{2} \int_0^\tau \int_0^{\infty}\int_0^{\infty} \mathfrak{K}_n(\varsigma, \varrho) \mathfrak{f}^n(\varsigma, \mathfrak{s}) \mathfrak{f}^n(\varrho, \mathfrak{s}) d\varrho d\varsigma d\mathfrak{s} =\int_0^{\infty} \mathfrak{f}^n_0 (\varsigma) d\varsigma.
\end{eqnarray*}
By using \eqref{trunker2}, we get
\begin{eqnarray*}
\int_0^{\infty} \mathfrak{f}^n (\varsigma, \tau) d\varsigma + \frac{1}{2}\int_0^\tau \bigg(\int_0^{\infty} r_n(\varsigma) \mathfrak{f}^n(\varsigma, \mathfrak{s}) d\varsigma\bigg)^2 d\mathfrak{s} \le \int_0^{\infty} \mathfrak{f}^n_0 (\varsigma) d\varsigma.
\end{eqnarray*}
The above inequality along with \eqref{trunic3} implies \eqref{eql1.2} and \eqref{eql1.3}.
\end{proof}

\medskip

Before proceeding further, we would like to mention the given notations: for $n \ge 1$, $a \in (1, +\infty)$, $ \delta \in (0, +\infty)$, and $\tau \in [0, +\infty)$, we write
\begin{eqnarray*}
\vartheta_{a, \delta}^n(\tau)=\sup
\left\{
  \begin{array}{ll}
    \int_0^a \mathds{1}_A(\varsigma)\mathfrak{f}^n(\varsigma, \tau) d\varsigma \\
     A \mbox{ is measurable subset of } \mathbb{R}_+ \mbox{ with } |A| \le \delta
  \end{array}
\right \}.
\end{eqnarray*}
Here $\mathds{1}_A$ denotes the indicator function of $A$ and $|A|$ is Lebesgue measure of $A$.

\begin{lemma}\label{uni-int}
Let \eqref{ker2} holds. Also, consider $T \in (0, +\infty)$ and $a \in (1, +\infty)$ then for every $n \ge 1$, $ \tau\in [0,T]$ and $\delta \in (0, +\infty)$ the following holds true
\begin{eqnarray}\label{int1}
\vartheta_{a, \delta}^n(\tau) \le C_2(a,T) \bigg[\vartheta_{a, \delta}^n(0)+\frac{C_1}{a}\bigg].
\end{eqnarray}
\end{lemma}
\begin{proof}

We begin the proof by letting  $a \in (1, +\infty)$, $\delta \in (0, +\infty)$ and consider subset $A \subset \mathbb{R}_+$, which is measurable with $|A| \le \delta$. Then from \eqref{ker2}, \eqref{mtrun}, and Fubini's theorem, we obtain
\begin{eqnarray*}
\int_0^a \mathds{1}_A(\varsigma) \mathfrak{f}^n(\varsigma,\tau) d\varsigma &=&\int_0
^a \mathds{1}_A (\varsigma)\mathfrak{f}_0^n(\varsigma) d \varsigma \\
& & +\int_0^\tau \int_0^a \int_0^{\infty} \mathds{1}_A (\varsigma) \mathfrak{K}_n(\varsigma+\varrho, \varrho) \mathfrak{f}^n(\varsigma +\varrho, \mathfrak{s}) \mathfrak{f}^n (\varrho, \mathfrak{s}) d\varrho d\varsigma d\mathfrak{s} \\
& &-\int_0^\tau \int_0^a \int_0^{\infty} \mathds{1}_A (\varsigma) \mathfrak{K}_n(\varsigma, \varrho) \mathfrak{f}^n(\varsigma, \mathfrak{s}) \mathfrak{f}^n (\varrho, \mathfrak{s}) d\varrho d\varsigma d\mathfrak{s}\\
& \le & \vartheta_{a, \delta}^n(0) + \int_0^\tau \int_a^{\infty} \int_0^{\infty} \mathds{1}_A (\varsigma-\varrho) \mathfrak{K}_n(\varsigma, \varrho) \mathfrak{f}^n(\varsigma, \mathfrak{s}) \mathfrak{f}^n (\varrho, \mathfrak{s}) d\varrho d\varsigma d\mathfrak{s}\\
& &+ \int_0^\tau \int_0^a \int_0^{\varsigma} [\mathds{1}_A (\varsigma-\varrho)-\mathds{1}_A (\varsigma)-\mathds{1}_A (\varrho)] \mathfrak{K}_n(\varsigma, \varrho) \mathfrak{f}^n(\varsigma, \mathfrak{s}) \mathfrak{f}^n (\varrho, \mathfrak{s}) d\varrho d\varsigma d\mathfrak{s} \\
& &-\int_0^\tau \int_a^{\infty} \int_0^{a} \mathds{1}_A (\varsigma) \mathfrak{K}_n(\varsigma, \varrho) \mathfrak{f}^n(\varsigma, \mathfrak{s}) \mathfrak{f}^n (\varrho, \mathfrak{s}) d\varrho d\varsigma d\mathfrak{s}.
\end{eqnarray*}
Set
\begin{eqnarray}\label{boundker1}
\sup_{(\varsigma, \varrho) \in (0,a)^2}\mathfrak{K}_n(\varsigma, \varrho) \le \mathit{K},
\end{eqnarray}
and
\begin{eqnarray}\label{boundker2}
\sup_{\varsigma \ge a} \ \sup_{\varrho \in (0,a)} \frac{\mathfrak{K}_n(\varrho, \varsigma)}{\varsigma} = \Theta(a),
\end{eqnarray}
where $\Theta(a) \to 0$ as $a \to \infty$. Then, we obtain
\begin{eqnarray*}
\int_0^a \mathds{1}_A(\varsigma) \mathfrak{f}^n(\varsigma,t) d\varsigma &\le & \vartheta_{a, \delta}^n(0)+\frac{C_1}{a} \\
& & +\mathit{K} \int_0^\tau \bigg[ \int_0^a \mathds{1}_{A+\varrho} (\varsigma) \mathfrak{f}^n(\varsigma,\mathfrak{s}) d\varsigma \int_0^{\varsigma} \mathfrak{f}^n(\varrho,\mathfrak{s}) d\varrho \bigg] d\mathfrak{s} \\
& & + \Theta(a) \int_0^\tau \bigg[ \int_a^{\infty}  \varsigma \mathfrak{f}^n(\varsigma,\mathfrak{s}) d\varsigma \int_0^{a} \mathds{1}_{\varsigma-A} (\varrho)\mathfrak{f}^n(\varrho,\mathfrak{s}) d\varrho \bigg] d\mathfrak{s} .
\end{eqnarray*}
Since
$$(A+\varrho)\cap (0, a) \subset (0,a) \ \ \mbox{and} \ \ |(A+\varrho)\cap (0, a)| \le |A+\varrho|=|A| \le \delta,$$
$$(\varsigma -A)\cap (0, a)\subset (0,a) \ \ \mbox{and} \ \ |(\varsigma -A)\cap (0, a)| \le |\varsigma -A|=|A| \le \delta,$$ then it follows from \eqref{masstrun} and \eqref{eql1.2} that
\begin{eqnarray*}
\int_0^a \mathds{1}_A(\varsigma) \mathfrak{f}^n(\varsigma,\tau) d\varsigma &\le & \vartheta_{a, \delta}^n(0)+\frac{C_1}{a}+ [\mathit{K}+\Theta(a)]C_3 \int_0^\tau \vartheta_{a, \delta}^n(\mathfrak{s}) d\mathfrak{s},
\end{eqnarray*}
 where $C_3=2\max\{\mathbb{G}_0, M_1^n(0)\}$. Hence, by Gronwall lemma we get
\begin{eqnarray*}
\vartheta_{a, \delta}^n(\tau) \le \bigg[\vartheta_{a, \delta}^n(0)+\frac{C_1}{a}\bigg]\int_0^\tau [\mathit{K}+\Theta(a)]C_3 d\mathfrak{s}.
\end{eqnarray*}
which implies \eqref{int1} with $C_2(a,T) \equiv [\mathit{K}+\Theta(a)]TC_3$.
\end{proof}

\medskip

\begin{lemma}\label{equicontinuity}
Let $T\in (0, +\infty)$ and $a \in (1, +\infty)$. For each $n \ge 1$ and $0 \le \mathfrak{s} \le \tau \le T$ following inequality holds
\begin{eqnarray}\label{cont1}
\int_0^a |\mathfrak{f}^n(\varsigma ,\tau)-\mathfrak{f}^n(\varsigma ,\mathfrak{s})| d\varsigma \le C_4(a)(\tau-\mathfrak{s}).
\end{eqnarray}
\end{lemma}
\begin{proof}
Let us put $\phi(\varsigma)= \mathds{1}_{[0,a]}(\varsigma) \sign(\mathfrak{f}^n(\varsigma,\tau)-\mathfrak{f}^n(\varsigma,\mathfrak{s}))$ in \eqref{wtrun} to obtain
\begin{eqnarray*}
\int_0^a|\mathfrak{f}^n(\varsigma,\tau)-\mathfrak{f}^n(\varsigma ,\mathfrak{s})| d\varsigma = \int_{\mathfrak{s}}^\tau \int_0^{\infty} \int_0^{\varsigma} B(\varsigma,\varrho,s_1) \mathfrak{K}_n(\varsigma, \varrho) \mathfrak{f}^n(\varsigma, s_1) \mathfrak{f}^n (\varrho, s_1) d\varrho d\varsigma ds_1,
\end{eqnarray*}
where
\begin{eqnarray*}
 B(\varsigma,\varrho,s_1)&=& \mathds{1}_{[0,a]} (\varsigma-\varrho)\sign(\mathfrak{f}^n(\varsigma-\varrho,\tau)-\mathfrak{f}^n(\varsigma-\varrho,\mathfrak{s}))\\
 & &-\mathds{1}_{[0,a]} (\varsigma)\sign(\mathfrak{f}^n(\varsigma,\tau)-\mathfrak{f}^n(\varsigma,\mathfrak{s}))-\mathds{1}_{[0,a]} (\varrho)\sign(\mathfrak{f}^n(\varrho,\tau)-\mathfrak{f}^n(\varrho,\mathfrak{s})).
\end{eqnarray*}
Now, by symmetric property of $\mathfrak{K}_n$ and Fubini's theorem, we obtain
\begin{eqnarray*}
\int_0^a|\mathfrak{f}^n(\varsigma,\tau)-\mathfrak{f}^n(\varsigma ,\mathfrak{s})| d\varsigma &\le & \int_{\mathfrak{s}}^\tau \bigg[\int_0^{\infty} \int_0^{\infty} \mathfrak{K}_n(\varsigma, \varrho) \mathfrak{f}^n(\varsigma, s_1) \mathfrak{f}^n (\varrho, s_1) d\varrho d\varsigma \\
& &+2\int_0^{a} \int_0^{\infty} \mathfrak{K}_n(\varsigma, \varrho) \mathfrak{f}^n(\varsigma, s_1) \mathfrak{f}^n (\varrho, s_1) d\varrho d\varsigma\bigg] ds_1\\
&\le & \int_{\mathfrak{s}}^\tau \bigg[\underbrace{\int_0^{a} \int_0^{a} \mathfrak{K}_n(\varsigma, \varrho) \mathfrak{f}^n(\varsigma, s_1) \mathfrak{f}^n (\varrho, s_1) d\varrho d\varsigma}_{=:I_1(a,s_1)}\\
& &+\underbrace{\int_a^{\infty} \int_a^{\infty} \mathfrak{K}_n(\varsigma, \varrho) \mathfrak{f}^n(\varsigma, s_1) \mathfrak{f}^n (\varrho, s_1) d\varrho d\varsigma}_{=:I_2(a,s_1)} \\
& &+4\underbrace{\int_0^{a} \int_0^{\infty} \mathfrak{K}_n(\varsigma, \varrho) \mathfrak{f}^n(\varsigma, s_1) \mathfrak{f}^n (\varrho, s_1) d\varrho d\varsigma}_{=:I_3(a,s_1)}\bigg] ds_1.\\
\end{eqnarray*}

By using \eqref{boundker1}, we obtain the following bound on $I_1(a,s_1)$
$$I_1(a,s_1) \le \mathit{K} [\mathbb{G}_0]^2.$$
 Next, for $a \ge 1$, it can be easily follows from \eqref{trunker1}-\eqref{alpha1} that 
 $$\int_{\mathfrak{s}}^\tau I_2(a, s_1) ds_1 \le (1+A) \frac{C_1}{a}.$$
 Now, with the help of \eqref{boundker2} we get
 $$I_3(a,s_1) \le \Theta(a) [\mathbb{G}_0]^2.$$
 
Finally, by using bounds on $(I_i)_{1\le i \le 3}$, we conclude
$$\int_{\mathfrak{s}}^\tau I_2(a, s_1) ds_1 \le \mathit{K} [\mathbb{G}_0]^2(\tau-\mathfrak{s})+(1+A) \frac{C_1}{a}+\Theta(a) [\mathbb{G}_0]^2 (\tau-\mathfrak{s}),$$
 which complete the proof of Lemma \ref{equicontinuity} with constant $C_4$ which depends on $a$, $\mathit{K}$, $\mathfrak{f}_0^n$ and $A$.
\end{proof}

\medskip
\subsection{Weak compactness}
In this subsection, we establish appropriate bounds for the application of the Dunford-Pettis Theorem, followed by the equicontinuity of the sequence $(\mathfrak{f}^n)_{n\ge 1}$, for the Arzel\`{a}-Ascoli Theorem \cite[Theorem 1.3.2]{Vrabie}, which says that, we only need to make sure that the sequence $(\mathfrak{f}^n)_{n\ge 1}$, has two properties given below:\\

\begin{enumerate}
\item[(P1)] For each $\tau\in [0,T]$, the set $\{\mathfrak{f}^n(\tau), \ n\ge 1\}$ is weakly compact in $L^1(\mathbb{R}_+)$.\\

\item[(P2)] The set $\{\mathfrak{f}^n, \ n\ge 1\}$ is weakly equicontinuous in $L^1(\mathbb{R}_+)$ at every $\tau\in [0,T]$ (see \cite[Definition 1.3.1]{Vrabie}).
\end{enumerate}

For proving (P1), first fix $\tau \in [0,T]$. For each $\epsilon \in (0, +\infty)$ there exists $R_{\epsilon}$ (large enough) such that $$\frac{\mathbb{G}_0}{R_{\epsilon}} \le \frac{\epsilon}{2}.$$ Then for $n \ge 1$, it follows from \eqref{trunic3} and \eqref{masstrun} that
\begin{eqnarray}\label{dunfii}
\int_{R_{\epsilon}}^{\infty} \mathfrak{f}^n(\varsigma,\tau) d\varsigma \le \frac{\mathbb{G}_0}{R_{\epsilon}} \le \frac{\epsilon}{2}.
\end{eqnarray}
Next, let us assume a measurable subset $A \subset \mathbb{R}_+$ with $|A| \le \delta $. Then, from \eqref{uni-int} and \eqref{dunfii}, we obtain
\begin{eqnarray*}
\int_A \mathfrak{f}^n(\varsigma,\tau) d\varsigma &\le & \int_0^{R_{\epsilon}} \mathds{1}_A \mathfrak{f}^n(\varsigma,\tau) d\varsigma +\frac{\epsilon}{2}\\
& \le & C_2(R_{\epsilon},T) \bigg[\vartheta_{R_{\epsilon}, \delta}^n(0)+\frac{C_1}{a}\bigg] +\frac{\epsilon}{2}.
\end{eqnarray*}
Now, by using uniform integrability of $\vartheta_{R_{\epsilon}, \delta}^n(0)$, we get for every $\epsilon >0$ there exists $\delta(\epsilon)$ such that
\begin{eqnarray}\label{uni-int1}
\sup_{n\ge 1} \int_A \mathfrak{f}^n(\varsigma,\tau) d\varsigma \le \epsilon, \ \ \ \ \mbox{provided} \ \ \ |A| \le \delta(\epsilon).
\end{eqnarray}
Finally, Owing to \eqref{masstrun}, \eqref{eql1.2}, \eqref{dunfii} and \eqref{uni-int1} property (P1) holds true as result of Dunford-Pettis theorem.

Next, we proceed to prove property (P2). For this, first let $\epsilon >0$ and $\psi \in L^{\infty}(\mathbb{R}_+)$ such that $\|\psi\|_{L^{\infty}_{L^{\infty}(\mathbb{R}_+)}} \le 1$. Since \eqref{dunfii} holds true for all $\tau \in [0,T]$ and $ n\ge 1$, then from \eqref{eql1.2} and \eqref{dunfii} there is $R_{\epsilon} \ge1$ such that
\begin{eqnarray*}
\int_{R_{\epsilon}}^{\infty} \psi(\varsigma) \mathfrak{f}^n(\varsigma,\tau) d\varsigma \le \frac{\|\psi\|_{L^{\infty}(\mathbb{R}_+)} \mathbb{G}_0}{R_\epsilon} \le \frac{\epsilon}{2}.
\end{eqnarray*}
now, by using Lemma \ref{equicontinuity} for $0 \le \mathfrak{s} \le \tau \le T$, we conclude
\begin{eqnarray}
\int_0^{\infty} |\psi(\varsigma)[\mathfrak{f}^n(\varsigma ,\tau)-\mathfrak{f}^n(\varsigma ,\mathfrak{s})]| d\varsigma &\le & \|\psi\|_{L^{\infty}(\mathbb{R}_+)}\int_0^{R_{\epsilon}} |\mathfrak{f}^n(\varsigma ,\tau)-\mathfrak{f}^n(\varsigma ,\mathfrak{s})| d\varsigma +\frac{\epsilon}{2}\nonumber\\
&\le & C_4(R_{\epsilon})(\tau-\mathfrak{s}) +\frac{\epsilon}{2} \le \epsilon, \label{cont2}
\end{eqnarray}
where $$|\tau-\mathfrak{s}| \le \frac{\epsilon}{2C_4(R_{\epsilon})}.$$

The time equicontinuity of the family $\{\mathfrak{f}^n(\tau), \tau\in [0,T]\}$ in $L^1(\mathbb{R}_+)$ has been implied from the estimate \eqref{cont2}.
Thus, a refined version of the Arzel\`{a}-Ascoli Theorem (see \cite[Theorem 2.1]{stewart1}) implies the existence of
a subsequence $(\mathfrak{f}^n)$ (not relabelled) and a function $\mathfrak{f} \in L^{\infty}((0,T); L^1(\mathbb{R}_+))$ such that
\begin{eqnarray}\label{wconv1}
\mathfrak{f}^n \to \mathfrak{f} ~~~ \text{in} ~~ C([0,T];L^1(\mathbb{R}_+)_w),
\end{eqnarray}
which means
\begin{eqnarray}\label{wconv2}
\lim_{n\to \infty} \sup_{\tau\in[0,T]}\bigg\{ \bigg| \int_0^{\infty} \{\mathfrak{f}^n(\varsigma,\tau)-\mathfrak{f}(\varsigma,\tau)\}
\psi(\varsigma) d\varsigma \bigg| \bigg\}=0,
\end{eqnarray}
for all $T>0$ and $\psi \in L^{\infty}(\mathbb{R}_+)$.
Non-negativity of $\mathfrak{f}^n(.,\tau), \forall~n \in \mathbb{N},$ implies that, for every $\tau \in [0, T],$
$$ \mathfrak{f}(.,\tau) \ge 0 ~~ \text{a.e. in} ~~\mathbb{R}_+.$$

 Finally, applying the weak convergence of $\{\mathfrak{f}^n(\tau)-\mathfrak{f}^n(\mathfrak{s})\}$ to $\{\mathfrak{f}(\tau)-\mathfrak{f}^n(\mathfrak{s})\}$ from \eqref{wconv1}, property (P1), and taking $\psi(\varsigma)=\sign (\mathfrak{f}^n(\varsigma,t_2)-\mathfrak{f}^n(\varsigma,t_1))$ in \eqref{cont2}, we conclude that
$$\| \mathfrak{f}(\tau)-\mathfrak{f}(\mathfrak{s})\|_{L^1(\mathbb{R}_+)} \le \epsilon.$$
Hence, we have
\begin{eqnarray} \label{space}
\mathfrak{f} \in \mathcal{C}([0, T]; L^1(\mathbb{R}_+)),
\end{eqnarray}
where $\mathcal{C}([0, T]; L^1(\mathbb{R}_+))$ is the space of all continuous functions from $[0,T]$ to $L^1(\mathbb{R}_+)$.

\subsection{Passing to the limit}
In this subsection, we provide arguments to support the claim that the limit function $\mathfrak{f}$ is, in fact, a weak solution to \eqref{crbk}-\eqref{ic} according to the concept of Definition \ref{defmild}. The following calculation is motivated from the \cite[Section 3]{Laurencot2}. Let us now consider $T \in (0,\infty )$, $a \in (0, \infty)$ and $M > a$. Then from \eqref{eql1.1} and property of $\Xi(\varsigma)$, we have for $n \ge M$
\begin{eqnarray*}
\int_0^T \bigg(\int_a^M r(\varsigma) \mathfrak{f}^n(\varsigma, \mathfrak{s}) d\varsigma \bigg)^2 d\mathfrak{s} \le \frac{C_1}{a}.
\end{eqnarray*}
Since $r \mathds{1}_{[a,M]} \in L^{\infty}(0, \infty)$, then it follows from \eqref{eql1.2}, \eqref{wconv1} and Lebesgue dominated convergence theorem that
\begin{eqnarray*}
\int_0^T \bigg(\int_a^M r(\varsigma) \mathfrak{f}(\varsigma, \mathfrak{s}) d\varsigma \bigg)^2 d\mathfrak{s} \le \frac{C_1}{a}.
\end{eqnarray*}
As $M >a$ is arbitrary then we get
\begin{eqnarray}\label{b1}
\int_0^T \bigg(\int_a^{\infty} r(\varsigma) \mathfrak{f}(\varsigma, \mathfrak{s}) d\varsigma \bigg)^2 d\mathfrak{s} \le \frac{C_1}{a}.
\end{eqnarray}
Finally, by using \eqref{masstrun}, \eqref{eql1.2}, \eqref{eql1.3} and \eqref{wconv1}, we conclude that for $\tau \in [0,T]$
\begin{eqnarray}
\sup_{t\in [0,T]}\|\mathfrak{f}(\tau)\|_{0,1} \le C_5, \label{norm}\\
\int_0^\tau \bigg(\int_0^{\infty} r(\varsigma) \mathfrak{f}( \varsigma, \mathfrak{s}) d\varsigma \bigg)^2 d\mathfrak{s} \le C_6, \label{bound3}\\
\int_0^{\infty} \varsigma \mathfrak{f}(\varsigma, \tau)d\varsigma \le \int_0^{\infty} \varsigma \mathfrak{f}_0(\varsigma) d\varsigma. \label{mass1}
\end{eqnarray}
Then it follows from \eqref{ker}, \eqref{ker1}, \eqref{norm}, \eqref{bound3} and Fubini theorem that
\begin{eqnarray}\label{space1}
(\varsigma ,\varrho, \mathfrak{s})\mapsto \mathfrak{K}(\varsigma,\varrho) \mathfrak{f}(\varsigma, \mathfrak{s}) \mathfrak{f}(\varrho,\mathfrak{s}) \in L^1(\mathbb{R}_+^2 \times (0,T)).
\end{eqnarray}

Now, let us consider a function $\phi \in L^{\infty}(\mathbb{R}_+)$ with $\|\phi \|_{L^{\infty}} \le 1$ and $\tau \in(0,\infty)$. Then from \eqref{trunic2} and \eqref{wconv1}, we obtain
\begin{eqnarray}\label{limit1}
\lim_{n \to \infty} \int_0^{\infty} \{\mathfrak{f}^n(\varsigma,\tau)-\mathfrak{f}_0^n(\varsigma)\} \phi(\varsigma) d\varsigma=\int_0^{\infty} \{\mathfrak{f}(\varsigma,\tau)-\mathfrak{f}_0(\varsigma)\} \phi(\varsigma) d\varsigma.
\end{eqnarray}
We next choose $a>1$, then for $n \ge 1$ and $0 \le s \le \tau$, the right hand side of \eqref{wtrun} can be written as sum of $\{B_{i,n}\}_{1 \le i \le 3}$ which are defined as
\begin{eqnarray*}
B_{1,n}(a,s)&:=& \int_0^a \int_0^{\varsigma} \tilde{\phi}(\varsigma, \varrho) \mathfrak{K}_n(\varsigma, \varrho) \mathfrak{f}^n(\varsigma, \mathfrak{s}) \mathfrak{f}^n(\varrho, \mathfrak{s}) d\varrho d\varsigma, \\
B_{2,n}(a, \mathfrak{s})&:=& \int_a^{\infty} \int_0^{a} \tilde{\phi}(\varsigma, \varrho) \mathfrak{K}_n(\varsigma, \varrho) \mathfrak{f}^n(\varsigma, \mathfrak{s}) \mathfrak{f}^n(\varrho,\mathfrak{s}) d\varrho d\varsigma, \\
B_{3,n}(a,\mathfrak{s})&:=& \int_a^{\infty} \int_a^{\varsigma} \tilde{\phi}(\varsigma, \varrho) \mathfrak{K}_n(\varsigma, \varrho) \mathfrak{f}^n(\varsigma, \mathfrak{s}) \mathfrak{f}^n(\varrho,\mathfrak{s}) d\varrho d\varsigma.
\end{eqnarray*}
Also, we define
\begin{eqnarray*}
B_{1}(a,\mathfrak{s})&:=& \int_0^a \int_0^{\varsigma} \tilde{\phi}(\varsigma, \varrho) \mathfrak{K}(\varsigma, \varrho) \mathfrak{f}(\varsigma, \mathfrak{s}) \mathfrak{f}(\varrho, \mathfrak{s}) d\varrho d\varsigma, \\
B_{2}(a,\mathfrak{s})&:=& \int_a^{\infty} \int_0^{a} \tilde{\phi}(\varsigma, \varrho) \mathfrak{K}(\varsigma, \varrho) \mathfrak{f}(\varsigma, \mathfrak{s}) \mathfrak{f}(\varrho, \mathfrak{s}) d\varrho d\varsigma, \\
B_{3}(a,\mathfrak{s})&:=& \int_a^{\infty} \int_a^{\varsigma} \tilde{\phi}(\varsigma, \varrho) \mathfrak{K}(\varsigma, \varrho) \mathfrak{f}(\varsigma, \mathfrak{s}) \mathfrak{f}(\varrho,\mathfrak{s}) d\varrho d\varsigma,
\end{eqnarray*}
here, $\tilde{\phi}$ is defined as in \eqref{wtrun}. For, $n \ge a$ and $(\varsigma, \varrho) \in (0,a)^2$, we have
$$\mathfrak{K}_n(\varsigma,\varrho)=\mathfrak{K}(\varsigma,\varrho).$$

Following this, we must demonstrate convergence of $B_{1,n}(a,s)$, $B_{2,n}(a,s)$ and $B_{3,n}(a,s)$ towards $B_{1}(a,s)$, $B_{2}(a,s)$ and $B_{3}(a,s)$, respectively. 

Thanks to \cite[Lemma 2.9]{Laurencot2}, we can easily obtain
\begin{eqnarray*}
\lim_{n \to +\infty} B_{1,n}( a, \mathfrak{s}) = B_{1}( a, \mathfrak{s}).
\end{eqnarray*}

The preceding identity, togather with \eqref{eql1.2} and the dominated convergence theorem, imply that
\begin{eqnarray}\label{limitB1}
\lim_{n \to +\infty} \int_0^\tau B_{1,n}( a, \mathfrak{s}) d\mathfrak{s}= \int_0^\tau B_{1}( a, \mathfrak{s}) d\mathfrak{s}.
\end{eqnarray}
Next, by using \eqref{masstrun}, \eqref{eql1.2} and \eqref{boundker2}, we obtain
\begin{eqnarray*}
\int_0^\tau B_{2,n}(a,\mathfrak{s}) d\mathfrak{s} &\le & \|\phi\|_{L^{\infty}(\mathbb{R}_+)} \int_0^\tau \int_a^{\infty} \int_0^{a} \mathfrak{K}_n(\varsigma, \varrho) \mathfrak{f}^n(\varsigma, \mathfrak{s}) \mathfrak{f}^n(\varrho,\mathfrak{s}) d\varrho d\varsigma d\mathfrak{s} \\
& \le & C_{7} \Theta(a) \|\phi\|_{L^{\infty}(\mathbb{R_+})}.
\end{eqnarray*}
Now, it follows from \eqref{trunker2} and \eqref{eql1.1} that
\begin{eqnarray*}
\int_0^\tau B_{3,n}(a,\mathfrak{s}) ds & \le & (1+A) \|\phi \|_{L^{\infty}(\mathbb{R}_+)} \int_0^\tau \bigg( \int_a^{\infty} r_n(\varsigma) \mathfrak{f}^n(\varsigma, \mathfrak{s}) d\varsigma \bigg)^2 d\mathfrak{s}  \\
&\le & \frac{C_7 \|\phi \|_{L^{\infty}(\mathbb{R}_+)}}{a}.
\end{eqnarray*}
Then, we get
\begin{eqnarray}\label{limit2}
\int_0^\tau |B_{2,n}(a,\mathfrak{s})+B_{3,n}(a,\mathfrak{s})| ds \le  C_8 \|\phi \|_{L^{\infty}(\mathbb{R}_+)}(a^{-1} + \Theta(a)).
\end{eqnarray}
In the same way we can from \eqref{ker}, \eqref{ker1}, and \eqref{b1} that
\begin{eqnarray}\label{limit3}
\int_0^\tau |B_{2}(a,\mathfrak{s})+B_{3}(a,\mathfrak{s})| d\mathfrak{s} \le  C_{9} \|\phi \|_{L^{\infty}(\mathbb{R}_+)}(a^{-1} + \Theta(a)).
\end{eqnarray}
Therefore, \eqref{limitB1}-\eqref{limit3} implies that
\begin{eqnarray}\label{limit4}
\limsup_{n \to +\infty} \bigg|\int_0^\tau \sum_{i=1}^{3}(B_{i,n}(a,\mathfrak{s})-B_{i}(a,\mathfrak{s})) d\mathfrak{s} \bigg| \le C_{10}\|\phi \|_{L^{\infty}(\mathbb{R}_+)}(a^{-1} + \Theta(a)).
\end{eqnarray}
Since left hand side of \eqref{limit4} is independent of $a >0$ and $\lim_{a \to \infty} \Theta(a) =0$, so we obtain
\begin{eqnarray}\label{limit5}
& & \hspace{-3cm}\lim_{n \to +\infty} \int_0^\tau \int_0^{\infty}\int_0^{\varsigma} \tilde{\phi}(\varsigma, \varrho) \mathfrak{K}_n(\varsigma, \varrho) \mathfrak{f}^n(\varsigma, \mathfrak{s}) \mathfrak{f}^n (\varrho, \mathfrak{s}) d\varrho d\varsigma d\mathfrak{s}\nonumber \\
 &= & \int_0^\tau \int_0^{\infty}\int_0^{\varsigma} \tilde{\phi}(\varsigma, \varrho) \mathfrak{K}(\varsigma, \varrho) \mathfrak{f}(\varsigma, \mathfrak{s}) \mathfrak{f} (\varrho, \mathfrak{s}) d\varrho d\varsigma d\mathfrak{s}.
\end{eqnarray}
By letting $n \to \infty$ in \eqref{wtrun} and owing to \eqref{limit1} and \eqref{limit5} we conclude that
\begin{eqnarray*}
\int_0^{\infty} \phi(\varsigma)(\mathfrak{f}(\varsigma,\tau)-\mathfrak{f}_0(\varsigma)) d\varsigma &=& \int_0^\tau \int_0^{\infty}\int_0^{\varsigma} \tilde{\phi}(\varsigma, \varrho) \mathfrak{K}(\varsigma, \varrho) \mathfrak{f}(\varsigma, \mathfrak{s}) \mathfrak{f} (\varrho, \mathfrak{s}) d\varrho d\varsigma d\mathfrak{s}\\
&=& \int_0^{\infty} \phi(\varsigma) \bigg[\int_0^\tau \int_0^{\infty} \mathfrak{K}(\varsigma+\varrho,\varrho) \mathfrak{f}(\varsigma+\varrho,\mathfrak{s}) \mathfrak{f}(\varrho,\mathfrak{s}) d\varrho \\
& & - \int_0^\tau \int_0^{\infty} \mathfrak{K}(\varsigma,\varrho) \mathfrak{f}(\varsigma,\mathfrak{s}) \mathfrak{f}(\varrho,\mathfrak{s}) d\varrho\bigg] d\mathfrak{s} d\varsigma .
\end{eqnarray*}
Due to the fact that this equality holds true for each $\phi \in L^{\infty}(0, +\infty)$, we have demonstrated that $\mathfrak{f}$ satisfies Definition \ref{defmild}. Hence Theorem \ref{thmexist} proved.

\medskip

\textbf{Acknowledgements:} PV would like to thank the Council of Scientific and Industrial Research (CSIR), India for granting the Ph.D. fellowship through Grant No. 09/143(0901)/2017-EMR-I.

\end{document}